\documentclass[a4paper,reqno]{amsart}
\usepackage{amsmath}
\usepackage{amsthm}
\usepackage{amsfonts}
\usepackage{amssymb}
\usepackage{color}
\usepackage{mathrsfs}
\usepackage{graphicx}
\usepackage{booktabs}
\usepackage{longtable}

\numberwithin{equation}{section}

\theoremstyle{plain}
\newtheorem{theorem}{Theorem}

\theoremstyle{plain}

\theoremstyle{plain}

\theoremstyle{plain}
\newtheorem{lemma}{Lemma}

\theoremstyle{plain}

\theoremstyle{plain}
\newtheorem{conjecture}{Conjecture}

\theoremstyle{plain}

\theoremstyle{plain}
\newtheorem{prop}{Proposition}

\theoremstyle{plain}
\newtheorem*{definition}{Definition}

\theoremstyle{definition}

\theoremstyle{definition}

\theoremstyle{definition}

\theoremstyle{definition}

\DeclareMathOperator{\del}{Delete}
\DeclareMathOperator{\dels}{DeleteSave}
\DeclareMathOperator{\nd}{d}
\DeclareMathOperator{\nwd}{wd}

\renewcommand{\le}{\leqslant}
\renewcommand{\ge}{\geqslant}

\title{weak degeneracy of regular graphs}

\author{Yuxuan Yang}

\address{Department of Science, Beijing University of Posts and Telecommunications, Beijing, China}

\email{yangyx@bupt.edu.cn}



\begin{document}

\begin{abstract}
Motivated by the study of greedy algorithms for graph coloring, Bernshteyn and Lee introduced a generalization of graph degeneracy, which is called weak degeneracy. 
In this paper, we show the lower bound of the weak degeneracy for $d$-regular graphs is exactly $\lfloor d/2\rfloor +1$, which is tight. This result refutes the conjecture of Bernshteyn and Lee on this lower bound.
\end{abstract}

\maketitle

\thispagestyle{empty}

%
%

\section{Introduction}\label{sec1}
In this paper, all graphs are finite and simple. Let $G$ be a graph. We denote by $V(G)$ and $E(G)$ the vertex set and the edge set of the graph $G$, respectively. Let $S$ be a subset of $V(G)$. Denote by $G-S$ the subgraph obtained from $G$ by deleting the vertices in $S$ together with their incident edges. Especially, when $S=\{v\}$ we may use $G-v$ instead of $G-\{v\}$. Denote by $G[S]$ the subgraph of $G$ induced by the vertex set $S$. Denote by $\deg_G(v) $ the degree of a vertex $v$ in a graph $G$. If $\deg_G(v)=d$ for each vertex $v$ in the graph $G$, we call $G$ a $d$-regular graph.

Recall that for a graph $G$, $\chi (G)$ denotes its chromatic number, that is, the minimum number of colors necessary to color the vertices of $G$ so that adjacent vertices are colored differently. A well-studied generalization of graph coloring is list coloring, which was introduced independently by Vizing \cite{vizing} and Erdős et al. \cite{er79}. There are further generalizations, such as DP-chromatic number \cite{dp18}, and DP-paint number \cite{KKLZ20}. As we shall not discuss these parameters, other than saying that they are bounded by the weak degeneracy defined below, we omit the definitions and refer the reader to \cite{KKLZ20} for the definitions and discussion about these parameters.

For a graph $G$, the greedy coloring algorithm colors vertices one by one in order $v_1, v_2,\dots, v_n$, assigning $v_i$ the least-indexed color not used on its colored neighbors. An upper bound for the number of colors used in such a coloring is captured in the notion of graph degeneracy. This greedy algorithm works for all variations of chromatic numbers mentioned above, and it makes graph degeneracy an important graph parameter. 

Frequently, graph degeneracy and chromatic numbers have a big gap between them. For example, the degeneracy of $d$-regular graph is exactly $d$, and its chromatic number varies from 2 to $d+1$. It is therefore interesting to see if we can modify the greedy coloring procedure to “save” some of the colors and get a better bound. For example, when we color a vertex under the greedy algorithm, we can choose a color that one of its neighbors can not use if possible. Motivated by this, Bernshteyn and Lee \cite{bl23} introduced a new graph parameter of \textbf{weak degeneracy}. Weak degeneracy is a variation of degeneracy which shares many nice properties of degeneracy. Since then weak degeneracy has been paid a lot of attention (see for example \cite{Han2023, Wang2023, Wang2021, Zhou2023}). Here is the definition and notation.

\begin{definition}[Delete Operation]
Let $G$ be a graph and let $f:V(G)\rightarrow \mathbb{N}$ be a function. For a vertex $u\in V(G)$, the operation
\begin{equation*}
\del (G,f,u)    
\end{equation*}
outputs the graph $G^{\prime}=G-u$ and the function $f^{\prime} :V(G^{\prime})\rightarrow \mathbb{Z}$ given by
\begin{equation}
\label{eq1.1}
   f^{\prime}(v)=
\begin{cases}
f(v)-1, & \mbox{if } uv\in E(G);\\
f(v), & \mbox{otherwise}.
\end{cases}
\end{equation}
An application of the operation $\del$ is \textbf{legal} if the resulting function $f^{\prime}$ is nonnegative.
\end{definition}

\begin{definition}[DeleteSave Operation]
Let $G$ be a graph and let $f:V(G)\rightarrow \mathbb{N}$ be a function. For a pair of adjacent vertices $u,w\in V(G)$, the operation
\begin{equation*}
\dels (G,f,u, w)    
\end{equation*}
outputs the graph $G^{\prime}=G-u$ and the function $f^{\prime} :V(G^{\prime})\rightarrow \mathbb{Z}$ given by
\begin{equation}
\label{eq1.2}
   f^{\prime}(v)=
\begin{cases}
f(v)-1, & \mbox{if } uv\in E(G) \mbox{ and }v\neq w;\\
f(v), & \mbox{otherwise}.
\end{cases}
\end{equation}
An application of the operation $\dels$ is \textbf{legal} if $f(u)>f(w)$ and the resulting function $f^{\prime}$ is nonnegative.
\end{definition}

A graph $G$ is \textbf{f-degenerate} if it is possible to remove all vertices from $G$ by a sequence of legal applications of the operations $\del$. A graph $G$ is \textbf{weakly f-degenerate} if it is possible to remove all vertices from $G$ by a sequence of legal applications of the operations $\dels$ and $\del$. Given $d\in \mathbb{N}$, we say that $G$ is \textbf{d-degenerate} if it is f-degenerate with respect to the constant function of value $d$. Also, we say that $G$ is \textbf{weakly d-degenerate} if it is weakly f-degenerate with respect to the constant function of value $d$. The \textbf{degeneracy} of $G$, denoted by $\nd(G)$, is the minimum integer $d$ such that $G$ is d-degenerate. The \textbf{weak degeneracy} of $G$, denoted by $\nwd(G)$, is the minimum integer $d$ such that $G$ is weakly d-degenerate.

Bernshteyn and Lee \cite{bl23} gave the following inequalities. 
\begin{prop}
\cite{bl23}
\label{prop1}
    For any graph $G$, we always have
    \begin{equation*}
        \chi (G)\le \chi _{l} (G)\le \chi _{DP} (G)\le \chi_{DPP}(G)\le \nwd(G)+1\le \nd (G)+1,
    \end{equation*}
\end{prop}
where $\chi_{DP}(G)$ is the DP-chromatic number of $G$, and $\chi_{DPP}(G)$ is the DP-paint number of $G$.

From Proposition \ref{prop1}, $\nwd (G) + 1$ is an upper bound for many graph coloring parameters. This is actually a direct result from the greedy coloring algorithm. Then, it is interesting to determine the weak degeneracy of a graph.

As far as lower bounds on weak degeneracy are concerned, by a double counting argument, Bernshteyn and Lee \cite{bl23} gave the following result.
\begin{prop}
\cite{bl23}
\label{thm0}
Let G be a $d$-regular graph with $n\ge 2$ vertices. Then $\nwd(G)\ge d-\sqrt{2n}$.
\end{prop}
In particular, if $n= O(d)$, then $\nwd(G)\ge d-O(\sqrt{d})$. Note that the upper bound $d$ is trivial and can be reached by the complete graph. It indicates that the weak degeneracy of a $d$-regular graph seems to be close to $d$. Then, they gave a conjecture about this.

\begin{conjecture}
\cite{bl23}
\label{conj}
Every $d$-regular graph $G$ satisfies $\nwd(G)\ge d-O(\sqrt{d})$.
\end{conjecture}

In this paper, we will refute this conjecture by constructing a $d$-regular graph with weak degeneracy $\lfloor d/2\rfloor +1$. 
\begin{theorem}
\label{thm1}
    There exists a $d$-regular graph with weak degeneracy $\lfloor d/2\rfloor +1$.
\end{theorem}

Also, we will show this is best possible by proving $\lfloor d/2\rfloor +1$ is exactly the lower bound.
\begin{theorem}
\label{thm2}
    Every $d$-regular graph $G$ satisfies $\nwd(G)\ge \lfloor d/2\rfloor +1$.
\end{theorem}
Since the degeneracy of $d$-regular graph is exactly $d$ and the weak degeneracy can be significantly smaller, we can say that in some cases the weak degeneracy of regular graphs is a better bound for those graph coloring parameters under Proposition \ref{prop1}.
%
%
\section{lower bound of weak degeneracy}
By a counting argument, we prove Theorem \ref{thm2} first.
\begin{proof}[Proof of Theorem \ref{thm2}]
    Suppose $G$ is a $d$-regular graph with $\nwd(G)\le \lfloor d/2 \rfloor$. Let $k=\nwd(G)$. By the definition of weak degeneracy, if we start with the constant function $f_1$ of value $k$ on $G$, there exists a sequence of legal applications of the operations $\del$ and $\dels$, which removes all vertices from $G$. We label the vertices 
    \begin{equation*}
        v_1,v_2,\dots , v_n
    \end{equation*}
    by the removing order, and let
    \begin{equation*}
        f_1, f_2,\dots , f_n \quad \mbox{ and } \quad G_1,G_2,\dots G_n
    \end{equation*} 
    denotes the corresponding functions and graphs during this procedure. To the specific, $f_i$ is the function after removing $v_1,v_2,\dots, v_{i-1}$ and it is a function on $G_i=G[\{v_{i},v_{i+1},\dots, v_n\}]$. We are interested in the value of the total sum of each function $f_i$, denoted by
    \begin{equation*}
        s_i:=\sum_{v\in G_i}f_i(v)=\sum_{j=i}^n f_i(v_{j}).
    \end{equation*}
    At the same time, the operation $\dels$ plays an important role in this procedure. We need two sequences to capture its effect. Let
    \begin{equation*}
           x_i=
    \begin{cases}
    1, & \mbox{if the operation of removing } v_i \mbox{ is }\dels;\\
    0, & \mbox{if the operation of removing } v_i \mbox{ is }\del.
    \end{cases}
    \end{equation*}
    Let $y_j$ denotes the number of $i$ that the operation $\dels(G_i,f_i,v_i,v_j)$ we used. It is not hard to find that 
    \begin{equation*}
        \sum_{i=1}^n x_i= \sum_{i=1}^n y_i,
    \end{equation*}
    since they both count the total number of $\dels$ we used. 

    Now, we consider $f_i(v_i)$. After $i-1$ operations, the function value decreased from $f_1(v_i)=k$ to $f_i(v_i)$. It means that $k-f_i(v_i)$ neighbors of $v_i$ were removed before $v_i$, which led to the decrease of the function value of $v_i$. Additionally, there are $y_i$ neighbors that were removed before $v_i$ without changing the function value of $v_i$. In total, $k-f_i(v_i)+y_i$ neighbors of $v_i$ were removed, then there are $d-(k-f_i(v_i)+y_i)$ remaining neighbors of $v_i$ in $G_i$. From \eqref{eq1.1} and \eqref{eq1.2}, we know the removal of $v_i$ makes that the total function value of $v_i$'s neighbors in $G_i$ decreases by $d-(k-f_i(v_i)+y_i)-x_i$.
    
    Comparing $s_i$ and $s_{i+1}$, we have 
    \begin{equation*}
        s_i-s_{i+1}=f_i(v_i)+d-(k-f_i(v_i)+y_i)-x_i.
    \end{equation*}
    Observe that $f_i(v_i)\ge x_i$. Since $k\le \lfloor d/2 \rfloor\le d/2$, we know
    \begin{equation}
    \label{eq2.1}
        s_i-s_{i+1}\ge d/2+x_i-y_i.
    \end{equation}
    Note that \eqref{eq2.1} works for $i\in \{1,2,\dots, n-1\}$. For $i=n$, it also works in principle, but we haven't defined $s_{n+1}$. Actually, it is natural to let $s_{n+1}=0$ since it is the total sum of a function with an empty domain. To avoid confusion, we investigate the case $i=n$ in detail. We have a single-vertex graph $G_n$ and we use $\del$ operation for $v_n$, then $x_n=0$ and $f_n(v_n)\ge 0$. Also, $k-f_n(v_n)+y_n$ has to be $d$, since $v_n$ has no neighbors in $G_n$. From, $y_n=d-k+f_n(v_n)\ge d/2$, we have 
    \begin{equation}
    \label{eq2.2}
        s_n=f_n(v_n)\ge 0 \ge d/2+x_n-y_n.
    \end{equation}
    Combining \eqref{eq2.1} and $\eqref{eq2.2}$, we know
    \begin{equation}
        s_1\ge nd/2+\sum_{i=1}^n x_i-\sum_{i=1}^n y_i=nd/2.
    \end{equation}
    On the other hand $s_1=kn\le nd/2$. It means that all the equal signs hold in the analysis above. Then, we have $f_i(v_i)=x_i$. However, $f_1$ is a constant function, so we can't use $\dels$ for $v_1$. Then, 
    \begin{equation*}
        f_1(v_1)=k\neq 0=x_1
    \end{equation*}
    gives a contradiction.
\end{proof}

%
%
\section{regular graphs with low weak degeneracy}
To prove Theorem \ref{thm1}, we start with the odd case. For each positive integer $k\ge 0$, we construct a graph $G$ such that $\nwd(G)\le k+1$ and $G$ is $(2k+1)$-regular.

Let $s>2k$ be a large integer to be defined and let $V_1$ be the disjoint union of  
\begin{equation*}
    \begin{aligned}
        &A=\{a_1,a_2,\dots,a_s\},\\
        &B=\{b_1,b_2,\dots,b_s\},\\
        &C=\{c_1,c_2,\dots,c_s\}.
    \end{aligned}
\end{equation*}
Here, $V_1=A\cup B\cup C$ is a subset of $V(G)$ and we assign the edges on $G[V_1]$ first. For each $i,j\in\{1,2,\dots,s\}$,
\begin{equation}
\label{edgev1}
\begin{aligned}
     &a_i, a_j \mbox{ are adjacent} \quad \Leftrightarrow \quad |i-j|\in[1,k-1],\\
    &b_i, b_j \mbox{ are adjacent} \quad \Leftrightarrow \quad |i-j|\in[1,k],\\
    &c_i, c_j \mbox{ are adjacent} \quad \Leftrightarrow \quad |i-j|\in[1,k],\\
    &a_i, b_i \mbox{ are adjacent} , \quad a_i, c_i \mbox{ are adjacent}.
\end{aligned}
\end{equation}

\begin{figure}[h!]
\centering
\includegraphics[width=0.99\linewidth]{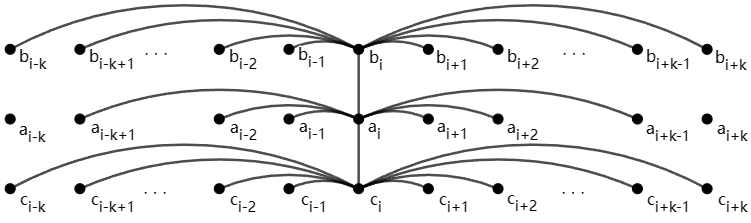}
\caption{ the main structure of $G(V_1)$}
\label{fig1}
\end{figure}

All the edges of $G[V_1]$ are given by \eqref{edgev1} and we denote this edge set by $E_1$. Figure \ref{fig1} shows the main structure of $G[V_1]$ by illustrating the neighbors of $a_i,b_i,c_i$. As you can see, in this structure $a_i,b_i$ and $c_i$ have $2k, 2k+1, 2k+1$ neighbors, respectively. Actually, this main structure only works for $i\in [k+1,s-k]$, and the structure is fragmentary when $i\le k$ or $i\ge s-k+1$. We collect all these abnormal vertices as
\begin{equation}
    D:=\{a_i|i\notin [k+1,s-k]\}\cup\{b_i|i\notin [k+1,s-k]\}\cup\{c_i|i\notin [k+1,s-k]\},
\end{equation}
and add pendent edges to them such that their degree is $2k+1$. In other words, we have a new vertex set $V_2$ of size
\begin{equation}
    |V_2|=\sum_{v\in D} (2k+1-\deg_{G[V_1]}(v)),
\end{equation}
and each $v\in D$ has $2k+1-\deg_{G[V_1]}(v)$ distinct neighbors in $V_2$. This edge set between $D$ and $V_2$ is denoted by $E_2$, which is shown in Figure \ref{fig2}.
\begin{figure}[h!]
\centering
\includegraphics[width=0.4\linewidth]{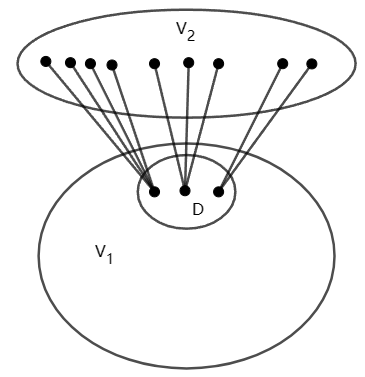}
\caption{ The edges between $D$ and $V_2$}
\label{fig2}
\end{figure}

The last step to construct a $(2k+1)$-regular graph is to add edges between $V_2$ and $V_1\backslash D$. By checking the current degree of each vertex in $V_1\cup V_2$, we can see that each $v\in V_2$ needs $2k$ more neighbors and $a_i$ needs one more neighbor for each $i\in[k+1,s-k]$. We simply choose a proper $s$ such that their demands match. A simple computation gives
\begin{equation}
    \begin{aligned}
        &|V_2|=6\cdot (1+2+\dots+k)=3k(k+1),\\
        &s=2k+2k|V_2|=6k^3+6k^2+2k.
    \end{aligned}
\end{equation}
To be specific, we have a map 
\begin{equation*}
    \phi: \{a_i|i\in[k+1,s-k]\}\rightarrow V_2
\end{equation*}
such that each vertex in $V_2$ has size-$2k$ preimage and we assign edges according to $\phi$, which means $v$ is adjacent to $\phi(v)$ for each $v$ in the domain. We denote the set of these edges by $E_3$. The construction of $G$ is done with $V(G)=V_1\cup V_2$ and $E(G)=E_1\cup E_2\cup E_3$. It is trivial to check that $G$ is $(2k+1)$-regular. 

We need to show that $\nwd(G)$ is at most $k+1$, so we start with a constant function of value $k+1$. We use $\del$ or $\dels$ in the following order on the vertice set:
\begin{equation*}
    a_1,b_1,c_1,a_2,b_2,c_2,a_3,b_3,c_3,\dots, a_s,b_s,c_s, V_2.
\end{equation*}
During this deletion process, we use $\dels(G,f,a_i,\phi(a_i))$ for $a_i\in V_1\backslash D$ when it is legal, and use $\del$ for all the other vertices. Note that $f$ represents the function on the vertex set at the current step of deletion, and it keeps changing during the process. We have the following observations.

\begin{lemma}
\label{lm2}
    Under the deletion order above, for each $v\in V_1$, at the step of deleting $v$, we have $f(v)\ge 0$.
\end{lemma}
\begin{proof}
    It suffices to show that there are at most $k+1$ neighbors that have been deleted at the step of deleting $v$. This is true since the deletion only happens in $V_1$ and the construction of $E_1$ in \eqref{edgev1} ensures this fact.
\end{proof}
Lemma \ref{lm2} shows that the deletions of $V_1$ are all legal.
\begin{lemma}
\label{lm1}
    Under the deletion order above, for each $a_i\in V_1\backslash D$, at the step of deleting $a_i$, we have $f(a_i)=2$.
\end{lemma}
\begin{proof}
    $a_i\notin V_2$ indicates that we never use $a_i$ as the second vertex of $\dels$. The neighbors
    \begin{equation*}
        a_{i-k+1},a_{i-k+2},\dots,a_{i-1}
    \end{equation*}
    of $a_i$ have been deleted and other neighbors have not. It means that
    \begin{equation*}
        f(a_i)=(k+1)-(k-1)=2.
    \end{equation*}
    at the current step.
\end{proof}
From Lemma \ref{lm1}, we know that we only use $\dels(G,f,a_i,\phi(a_i))$ for $a_i\in V_1\backslash D$ if $f(\phi(a_i))\le 1$ at the step of deleting $a_i$.

It remains to investigate the function values on $V_2$. Note that $V_2$ is an independent set of $G$.
\begin{lemma}
    Under the assumptions above, after the deletion of $V_1$, $f(v)\ge 0$ for all $v\in V_2$.
\end{lemma}
\begin{proof}
    For each vertex $v\in V_2$, one neighbor $u$ is in $D$ and the remaining neighbors form the preimage $\phi^{-1}(v)$ with size $2k$. From Lemma \ref{lm1}, the deletion of each vertex in $\phi^{-1}(v)$ can never reduce its function $f(v)$ to $0$ because $\dels$ is used when $f(v)=1$. With considering the deletion of $u$, the value of $v$ is still non-negative.
\end{proof}
This completes the proof of $\nwd(G)\le k+1$. From Theorem \ref{thm2}, we know $\nwd(G)\ge k+1$ since $G$ is $(2k+1)$-regular. As a result, the construction of $G$ establishes the odd case of Theorem \ref{thm1}.
\begin{theorem}
\label{thm3}
    There exists a $(2k+1)$-regular graph with weak degeneracy $k+1$.
\end{theorem}
It is not hard to generalize this construction to the even case, but actually the even case is a direct consequence of the odd case by the following argument.
\begin{theorem}
\label{thm4}
    If $G$ is $d$-regular, then there exists a $(d+1)$-regular graph $G^{\prime}$ such that $\nwd(G^{\prime})\le \nwd(G)+1$.
\end{theorem}
\begin{proof}
    We give the construction of $G^{\prime}$ explicitly. We collect $d+1$ copies of $G$ and add a common neighbor for $d+1$ copies of each vertex in $G$. See Figure \ref{fig3}.
\begin{figure}[h!]
\centering
\includegraphics[width=0.4\linewidth]{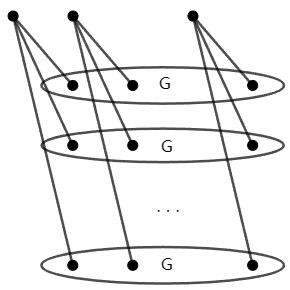}
\caption{The construction of $G^{\prime}$}
\label{fig3}
\end{figure}
    To be specific, let 
    \begin{equation*}
       V(G^{\prime}):=V(G)\times \{0,1,2,\dots, d+1\}.
    \end{equation*}
    For $i\in\{1,2,\dots, d+1\}$, $(u,i)$ and $(v,i)$ are adjacent if $uv\in E(G)$. Also, $(v,i)$ and $(v,0)$ are adjacent for each $i\in\{1,2,\dots, d+1\}$ and $v\in V(G)$. These are all the edges of $G^{\prime}$. It is trivial to check $G^{\prime}$ is $(d+1)$-regular.
    
    We need to show $\nwd(G^{\prime})\le \nwd(G)+1$, so we start with a constant function of value $\nwd(G)+1$. 
    We can use $\del$ operations for each $(v,0)\in V(G^{\prime})$. These vertices form an independent set and these deletion operations are legal. 
    The remaining graph is $d+1$ copies of $G$ and the function value is a constant $\nwd(G)$. By the definition of weak degeneracy, we can legally remove each copy of $G$ by $\del$ and $\dels$. As a result, we have $\nwd(G^{\prime})\le \nwd(G)+1$.

\end{proof}

We end with a proof of Theorem \ref{thm1}.
\begin{proof}[Proof of Theorem \ref{thm1}]
    When $d$ is an odd number, the statement holds because of Theorem \ref{thm3}. When $d$ is an even number, from Theorem \ref{thm3}, we have a $(d-1)$-regular graph with weak degeneracy $d/2$. By Theorem \ref{thm4}, we have a $d$-regular graph with weak degeneracy at most $d/2+1$. By Theorem \ref{thm2}, its weak degeneracy is at least $d/2+1$. Therefore, we have a $d$-regular graph with weak degeneracy $d/2+1$.
\end{proof}

\section*{Acknowledgements}
Research supported by ``the Fundamental Research Funds for the Central University'' in China.

%
%

\end{document}